\documentclass[12pt]{amsart}
\usepackage{a4}
\usepackage{amssymb}
\usepackage{amsmath}
\usepackage{amsthm}
\usepackage{amstext}
\usepackage{amscd,setspace}
\usepackage{latexsym}
\usepackage{graphics}
\theoremstyle{plain}
\newtheorem{thm}{Theorem}[section]

\newtheorem{lem}[thm]{Lemma}

\newtheorem{prop}[thm]{Proposition}
\theoremstyle{definition}
\newtheorem{rmk}[thm]{Remark}

\newtheorem{defn}[thm]{Definition}

\numberwithin{equation}{section}

\newcommand{\sD}{{\mathcal D}}

\newcommand{\sH}{{\mathcal H}}

\newcommand{\C}{{\mathbb C}}

\newcommand{\N}{{\mathbb N}}

\newcommand{\Q}{{\mathbb Q}}
\newcommand{\R}{{\mathbb R}}

\input{xy}
\xyoption{all}

\begin{document}

\title[]{On uniform lattices in real semisimple groups}

\author[C.Bhagwat]{Chandrasheel Bhagwat}
\address{Indian Institute of Science Education and Research, Pune-
  411008, India}

\email{cbhagwat@iiserpune.ac.in}
\author[S.Pisolkar]{Supriya Pisolkar}
\address{Indian Institute of Science Education and Research, Pune-
  411008, India}
\email{supriya@iiserpune.ac.in}


\date{\today}

\begin{abstract}
In this article we prove that the co-compactness of the arithmetic lattices in a connected
semisimple real Lie group is preserved if the lattices under consideration are
representation equivalent. This is in the spirit of the question posed by Gopal
 Prasad and A. S. Rapinchuk in \cite{PR1} where instead of representation equivalence,
the lattices under consideration are weakly commensurable Zariski dense subgroups.
\end{abstract}

\thanks{C.B. is partially supported by DST-INSPIRE Faculty scheme, award number [IFA- 11MA-05]}

\maketitle
\section{Introduction}
In \cite{PR}, G. Prasad and A. S. Rapinchuk defined the notion of weakly commensurable
Zariski dense subgroups in absolutely almost simple algebraic groups. Among many other
striking implications of this seemingly weak notion they have proved that weakly commensurable
 subgroups in the group of rational points of absolutely almost simple algebraic groups
determine the type of the group except in the case when one is of type B and the other of
type C. In this paper they show that length commensurable arithmetic lattices are weakly
commensurable. For this, when the locally symmetric spaces are of rank greater than 1, they
assume the validity of Schanuel's conjecture. Further using methods from arithmetic theory
of algebraic groups, they obtain commensurability type results for isospectral compact
locally symmetric spaces.

In \cite{BPR}, we assume a stronger hypothesis that the lattices defining the locally symmetric
 spaces are representation equivalent rather than isospectral on functions. This allowed us to
obtain similar conclusions as in \cite{PR} for representation equivalent lattices, without
invoking Schanuel's conjecture.

In the sequel to their work on weakly commensurable subgroups, in \cite{PR1} Gopal Prasad and
A.S. Rapinchuk have posed the following question. For $i=1,2$, let $G_i$ be a connected absolutely
 almost simple group defined over $F = \R$ or $\C$ and $\Gamma_i$ be a lattice in $G_i(F)$. Assume
that $\Gamma_1$ is weakly commensurable to $\Gamma_2$. Does compactness of $\Gamma_1 \backslash G_1(F)$
imply the compactness of $\Gamma_2 \backslash G_2(F)$? When the corresponding locally symmetric spaces
are length commensurable and one of the space is arithmetically defined, the results Theorem $6$ and
Theorem $7$ of \cite{PR} provide an affirmative answer to the above question. We recall that the co-compactness
of a lattice in a semisimple real Lie group is equivalent to the absence of nontrivial unipotents in it
(cf. \cite{R}, Corollary 11.13). Thus the above question can be rephrased as whether for two weakly commensurable
lattices, the existence of nontrivial unipotent elements in one of them implies their existence in the other.

In this article we address a similar question under the stronger hypothesis of representation equivalence.
We prove that:

\begin{thm}\label{intro-main}
Let $G$ be a connected semisimple real Lie group. Let $\Gamma_{1}$, $\Gamma_{2}$ be representation equivalent
arithmetic lattices in $G$. Then $\Gamma_1\backslash G$ is compact if and only if $\Gamma_{2}\backslash G$
is compact.
\end{thm}

\begin{rmk} By the arithmeticity theorem of Margulis, if $G$ is a real semisimple algebraic group
 without compact factors and such that $\R$-rank of $G$ is $\geq 2$, then every irreducible lattice is arithmetic.
If $\R$-rank is $1$, a result of Corlette in archimedean case and Gromov-Schoen in non-archimedean case
shows that lattices in $Sp(n, 1), n \geq 2$ and $F_4^{-20}$ are arithmetic.
\end{rmk}

\begin{rmk} For $p$-adic groups, Theorem \ref{intro-main} is a tautology since every lattice is co-compact.
\end{rmk}

\section{Preliminaries}

\subsection{Lattices and representation equivalence}

Let $G$ be a connected semisimple real Lie group. Suppose $\Gamma$ is a discrete subgroup of $G$ such that
the quotient $\Gamma \backslash G$ has a finite $G$-invariant Borel measure $\mu$.
Consider the space $L^{2}(\Gamma \backslash G)$ of all complex valued measurable $\Gamma$-invariant functions
on $G$ such that
\[ \int \limits_{\Gamma \backslash G} {\mid f(x)\mid}^2 d\mu(x) < \infty.\]
The right regular representation $R_{\Gamma}$ of $G$ is on $L^{2}(\Gamma \backslash G)$ defined by,
\[ R_{\Gamma}(g)f(x) = f(xg) \quad \forall~ g,x \in G ~\text{and}~ f\in  L^{2}(\Gamma \backslash G) \]
It is well known that this defines a unitary representation of $G$ on the Hilbert space $L^{2}(\Gamma \backslash G)$.

\smallskip

Let $\widehat{G}$ be the set of all equivalence classes of irreducible unitary representations of $G$.
We will denote an element of $\widehat{G}$ by $\omega$.
We now recall the following result (cf. \cite{W}, 14.10.5) which describes the direct integral decomposition of $R_{\Gamma}$
with respect to the irreducible unitary representations of the group $G$.

\begin{thm} \label{directintegral} Let $(\pi,\sH)$ be a unitary representation of $G$ on a Hilbert space $\sH$.
There exists a Borel measure $\sigma$ on $\widehat{G}$ and a family of unitary representations $(\pi_{\omega},~ H_{\omega})$
such that:\smallskip
\begin{enumerate}
\item The representation $(\pi,\sH)$ is unitarily equivalent to a direct integral as follows:
 \[(\pi,\sH) \cong  \int \limits_{\widehat{G}} (\pi_{\omega}, H_{\omega})~ d\sigma(\omega). \]
\item Each $(\pi_{\omega}, H_{\omega})$ is unitarily equivalent to the Hilbertian tensor product
 $(\pi^{'}_{\omega}\otimes I,~ H^{'}_{\omega} \otimes V_{\omega})$ of an irreducible unitary
 representation
 $(\pi^{'}_{\omega},H^{'}_{\omega}) \in \omega$ and the trivial $G$ representation $I$
    on some Hilbert space $V_{\omega}$.\smallskip
\item The map $\omega \mapsto \text{dim}(V_{\omega})$ is measurable w.r.t. the measure $\sigma$.\medskip
\end{enumerate}
\end{thm}

The following result from (cf. \cite{D}) gives the appropriate uniqueness for the measure $\sigma$ in
Theorem \ref{directintegral}.

\begin{prop}\label{absolute-measure} If there are two Borel measures $\sigma$ and $\mu$ on $\widehat{G}$ such that all the three
 conditions in
Theorem $\ref{directintegral}$  hold, then
$\sigma$ and $\mu$  are mutually absolutely continuous i.e., for any Borel set $E$,
 $$\sigma(E) = 0 \Leftrightarrow \mu(E) = 0.$$
\end{prop}
\vspace{2mm}

\subsection{Eisenstein series:}\label{eisenstein} In this subsection we recall some of the relevant facts from
the theory of Eisenstein series from Langlands' work ~ \cite{L1}, \cite{L2} and \cite{OW}. In particular we discuss the
decomposition of the Hilbert space $L^{2}(\Gamma \backslash G)$ into certain $G$-invariant
spaces parametrized by various parabolic subgroups.

Let $G$ be the group of real points of a connected semisimple  group $\mathbf{G}$ defined
over $\Q$ and let $\Gamma$ be a lattice in $G$ which we assume to be neat.
Let us fix a minimal parabolic subgroup $\mathbf{P}$ of $\mathbf{G}$ defined over $\Q$ and
a maximal $\Q$-split torus $\mathbf{A}$ of $\mathbf{P}$. A standard cuspidal parabolic subgroup
$P$ is  the normalizer of a parabolic subgroup $\mathbf{P}$.

Let $\mathfrak{a}_{\C}$ be the complexification of the Lie algebra of  $\mathbf{A}$. The set
$\mathfrak{a}$ of real points of $\mathfrak{a}_{\C}$ corresponding to the split component of $P$.
Consider a decomposition $P = AMN$ of $P$, where $A = \mathbf{A}_{\R}^{o}$ is the analytic subgroup
of $G$ with Lie algebra  $\mathfrak{a}$; $N$ is the Unipotent radical of $P$ and $M$ is a reductive
group identified with $N \backslash MN $. Since $\Gamma$ is neat,  $\Gamma \cap P \subseteq MN$ and
$\Theta: = \Gamma \cap N \backslash \Gamma \cap MN $ can be thought of as a subgroup of $N \backslash MN  \cong M$.
Let $S = MN$. Let $(P,S)$ and $(P',S')$ be two split parabolic subgroups of $G$. Then we say that $(P,S)$
is a successor of $(P', S')$ i.e. $(P,S) \geq (P',S')$ if
$P \supset P'$ and $S \supset S'$. Further $(P,S)$ is called as a dominant successor of $(P', S')$ if
there exists a chain \\
$$(P,S) = (P_1,S_1) \geq (P_2,S_2) \geq \cdots \geq(P_n,S_n) = (P',S')$$
such that $$P_1 \supseteq P_2 \supseteq \cdots \supseteq P_n$$
and $$A_1 \subseteq A_2 \subseteq A_2 \subseteq \cdots \subseteq A_n$$ and
${\text dim} (A_{i+1}) - {\text dim}(A_i) = 1, 1 \leq i \leq n$.

\begin{defn} A subgroup $(P,S)$ is said to be $\Gamma$-cuspidal if every dominant successor $(P',S')$ of $(P,S)$
has the following properties: \\
\noindent (1) $\Gamma \cap P'$ is contained in $S'$. \\
\noindent (2) $N'/N'\cap \Gamma$ is compact.\\
\noindent (3) $S'/S'\cap \Gamma$ is of finite volume.\\
If, moreover $S/S \cap \Gamma$ is compact then $(P,S)$ is said to be $\Gamma$ per-cuspidal.
\end{defn}

Let $E(G, \Gamma)$ denote the set of all $\Gamma$-percuspidal subgroups of $G$. We recall here the
important result about $E(G, \Gamma)$ ( cf. \cite{OW}, Proposition 2.6)

\begin{prop} Modulo $\Gamma$-conjugacy, there are only finitely many elements of $E(G, \Gamma)$.
\end{prop}

The number of cusps of $\Gamma$ is then by definition,  $ |(\Gamma \backslash E(G, \Gamma)) |$.

\begin{defn}Two cuspidal subgroups $P$ and $P'$ are said to be associate if there is an element of
the Weyl group of  which takes $\mathfrak{a}_C$ to $\mathfrak{a}'_C$.
\end{defn}

Consider a decomposition $P = AMN$ of $P$ as before and denote by $Z$ the center of universal enveloping algebra of $M$.
Let $V(\xi): = \left\{
\phi \in L^{2}_{0} (\Theta \backslash M): X \phi = \xi(X) \phi
\quad \forall ~ X \in Z \right\}$ for $\xi \in {\rm Hom}(Z,\C)$.\smallskip

Let $E$ be the set of all orbits of the action of $Z$ on ${\rm Hom}(Z,\C)$ and let
$V_{E}:= \bigoplus \limits_{\xi \in E} V(\xi)$. This is a closed $M$-invariant subspace of
$L^{2}_{0} (\Theta \backslash M)$ such that

$$ L^{2}_{0} (\Theta \backslash M) = \bigoplus \limits_{E} V_{E}.$$

Such a $V_{E}$ is called a simple admissible subspace of $L^{2}_{0} (\Theta \backslash M)$.\medskip

Fix such a simple admissible subspace $V$. Let $K$ be a maximal compact subgroup of $G$ and
 $W$ be the space spanned by the matrix coefficients of some irreducible representation of
 $K$.

  Let $\sD(V,W)$ be the space of all continuous functions $\phi$ on $N(\Gamma \cap P) \backslash G$
	such that
$m \mapsto\phi(mg)$ belongs to $V$ and $k \mapsto \phi(gk^{-1})$ belongs to $W$ for all $g \in G$ and
such that the support of $\phi$ on $NM \backslash G$ is compact.

Let $\left\{ P \right\}$ be an associate class of per-cuspidal parabolic subgroups of $G$. Define
$L(\left\{P\right\},\left\{V\right\},W)$
to be the closed subspace spanned by functions $\hat{\phi}$ with $\phi \in \sD(V(P),W)$
for some $P \in \left\{ P \right\}$.

\bigskip

From Lemma 2 in \cite{L1}, we know that:

 \begin{lem}\label{ortho-direct-sum} The space $L^{2}(\Gamma \backslash G)$ is the orthogonal
direct sum of the spaces as follows:

$$L^{2}(\Gamma \backslash G) = \bigoplus \limits_{\left\{P\right\}} L(\left\{P\right\},\left\{V\right\},W).$$

Further each $L(\left\{P\right\},\left\{V\right\},W)$ can be decomposed as:

\begin{equation}
L(\left\{P\right\},\left\{V\right\},W) =  \bigoplus \limits_{i=0}^{g} L_{i}(\left\{P\right\},\left\{V\right\},W)
\end{equation}

where $g$ is the common rank of all parabolic subgroups in the class $\left\{ P\right\}$.
\end{lem}

\begin{rmk}\label{imp}
\noindent (1) The important hypothesis about the lattice  $\Gamma$ for the above result as in \cite{L2} was that
$\Gamma$ possesses a fundamental domain.
It follows from the results of Raghunathan and Garland \cite{GR} in the rank one case and of
Margulis \cite{M} in the higher rank case that there exist fundamental domains for the arithmetic
lattices in $G$. Thus the hypotheses in the decomposition theorem of Langlands (as in \cite{L1}, \cite{L2})
are satisfied.\\

\noindent (2) In his result in \cite{L2}, Langlands considers a complete set $\mathcal{P}(G, \Gamma)$ of per-cuspidal subgroups
of $G$. It can be verified (cf. \cite{OW}, Page 78) that the set $E(G, \Gamma)$ is exactly the set $\mathcal{P}(G, \Gamma)$.
\end{rmk}

\section{Main results}
In this section we prove the main result of this article.

\begin{thm}\label{main}
Let $G$ be a connected semisimple real Lie group and $\Gamma_{1}$, $\Gamma_{2}$ be two arithmetic
lattices in $G$. If the lattices $\Gamma_{1}$, $\Gamma_{2}$ are representation equivalent and $\Gamma_1 \backslash G$ is
compact then $\Gamma_2 \backslash G$ is also compact.
\end{thm}

The main ingredient of the proof of this theorem is the following characterization of co-compact lattices
in real semisimple Lie groups.

 \begin{thm}\label{cocompactness} Let $\Gamma$ be an irreducible lattice in $G$. Then the quotient
$\Gamma \backslash G$ is  compact if and only if the direct integral decomposition given by the Theorem
\ref{directintegral} is a Hilbert direct sum  i.e.
 \[ R_{\Gamma} \cong \widehat{ \bigoplus \limits_{j \in \N} }(\pi_j, V_j)\]for a countable family
 of irreducible unitary representations $(\pi_j, V_j)_{j \in \N}$  such that each $\pi_j$
 occurs with a finite multiplicity.
\end{thm}

\begin{proof}
If $\Gamma \backslash G$ is compact then it is well known that
$R_{\Gamma}$ is a Hilbert direct sum as required. Conversely, if $\Gamma$ is not uniform,
then there is a unipotent element $u \neq 1$ in $\Gamma$  (cf. \cite{R}, Corollary 11.13).
To such a unipotent element $u$ in $\Gamma$ one can associate a proper parabolic subgroup of
$G$. Indeed, when $\mathbf{G}$ if of $\Q$-rank 1, then by the result ( Cf. 12.17, \cite{R}),
$u$ is contained in a unique $\Q$-parabolic subgroup of $\mathbf{G}$. In the case when
$\mathbb{Q}$-rank of $\mathbf{G}$ is atleast $2$, one associates to $u$ a $\Q$-parabolic subgroup
$P_u$ of $\mathbf{G}$ containing $u$ by following the procedure of Borel-Tits in \cite{BT}.
Let $U_1$ be the one parameter subgroup containing $u$, then take its normaliser $N_1$.
Let $U_2$ be the unipotent radical of $N_1$ (it contains $U_1$). Let $N_2$ be the normaliser
of $U_2$. After some stage this chain of $\Q$-subgroups $N_i$ and $U_i$ stabilizes. Thus,
we get the unipotent group $U=U_n$ which is the unipotent radical of the normaliser $P=N_n$,
and $U_n=U_{n+1}$. Then a result of Borel-Tits (cf. \cite {BT})
says that $P_u:=P$ is a proper parabolic subgroup containing $U \supset U_1$ and  $u \in U_1$. \\ \\
Let $E(G, \Gamma)$ be the set of all per-cuspidal parabolic subgroups of $G$ as in \ref{eisenstein}.
Since $\Gamma$ is arithmetic, $P_u \in E(G, \Gamma)$ ( cf. \cite{OW}, P. 23, 63).
Thus $P_u$ appears in the decomposition in Lemma \ref{ortho-direct-sum}. It follows from
(cf. \cite{L1}, P. 254) that the space $L_{i}(\left\{ P\right\}, \left\{ V \right\}, W)$ has a
continuous spectrum of dimension $i$. Thus we conclude that there is a non-trivial continuous
spectrum in the above decomposition since $P$ is proper. Hence the result follows.
\end{proof}

\noindent We now give the proof of the main theorem.

\begin{proof}[Proof of the Theorem \ref{main}]

The lattice $\Gamma_1$ is co-compact so by Theorem \ref{directintegral} there is a countable
subset $E$ of $\widehat{G}$ such that support of $\mu_{1}$ equals $E$. (Recall that
support of a measure is the set of all points $\omega$ in $\widehat{G}$ for which
every open neighborhood $U$ of $\omega$ has positive measure.)

Let $\mu_1$ and $\mu_2$ be the measures on $\widehat{G}$ corresponding to the representations
$R_{\Gamma_{1}}$ and $R_{\Gamma_{2}}$, respectively. Hence the measures $\mu_1$ and $\mu_2$ are
mutually absolutely continuous by the Proposition \ref{absolute-measure}. It follows that their
supports are equal. Thus the support of measure $\mu_2$ also equals $E$. In other words, $R_{\Gamma_{2}}$
has a decomposition as a Hilbert direct sum of irreducible representations of $G$. From Theorem
\ref{cocompactness} the desired result follows.

\end{proof}

{\small {\it Acknowledgements:}
We thank M S. Raghunathan, Nolan Wallach,  Laurent Clozel, David Vogan,
Paul Garrett, A. Raghuram and Sandeep Varma for their interest and helpful discussions. Our sincere thanks
to Gopal Prasad for pointing out the inaccuracies in the previous version.}


\begin{thebibliography}{1000000}

\bibitem{BPR} Bhagwat, C.; Pisolkar, S.; Rajan, C. S. {\em Commensurability and representation equivalent arithmetic lattices}, Int. Math. Res. Not., {\bf no. 8} (2014) 2017--2036.

\bibitem {BT} Borel, A.; Tits, J. {\em \'El\'ements unipotents et sous-groupes paraboliques de groupes r\'eductifes. I}, {\it Invent. Math.} {\bf 12} (1971), 95--104.

\bibitem{D} Dixmier J., {\em Les $\C^{\ast}$-algebras et leurs representations}, Gauthier-Villarrs, Paris, 1969.

\bibitem{GR} Garland H. and Raghunathan M. S., {\em Fundamental Domains for Lattices in (R-)rank 1 Semisimple Lie Groups}, {\it Annals of Mathematics}, Second Series, {\bf 92}, No. 2 (1970), 279--326.

\bibitem{L1} Langlands R., {\em Eisenstein Series},
Proceedings of the Symposium in Pure Mathematics of the American Mathematical Society held at The University of Colorado Boulder, Colorado, (1965), 235--252.

\bibitem{L2} Langlands R., {\em On the Functional Equations Satisfied by Eisenstein Series},  Springer --Verlag Lecture Notes in Math., {\bf 544}, Springer--Verlag, Berlin–Heidelberg,
New York, (1976), 1--337.

\bibitem{M} Margulis G., {\em On the arithmeticity of discrete groups },  Soviet Math. Dokl. {\bf 10} (1969) 900–-902.

\bibitem{OW} Osborne M.; Warner G., {\em The Theory of Eisenstein Systems}, Academic Press, London, 1981.

\bibitem{PR} Prasad, G.; Rapinchuk, A. S., {\em Weakly commensurable
arithmetic groups and isospectral locally symmetric spaces},
Publ. Math. Inst. Hautes Études Sci. {\bf 109} (2009), 113--184.

 \bibitem{PR1}  Prasad, G.; Rapinchuk, A. S., {\em Generic elements in Zariski-dense subgroups and isospectral locally symmetric spaces. Thin groups and superstrong approximation}, Math. Sci. Res. Inst. Publ., {\bf 61}, Cambridge Univ. Press, Cambridge, (2014) 211--252.

\bibitem{R} Raghunathan M. S., {\em Discrete Subgroups of Lie groups} Springer, Berlin, 1972.


\bibitem{W} Nolan Wallach, {\em Real Reductive Groups II}, Academic Press, Boston, 1992.



\end{thebibliography}
\end{document}